\documentclass[]{theclass}
\begin{document}


\begin{frontmatter}   

\titledata{Saved by the rook: a case of matchings \\and Hamiltonian cycles} {}          

\authordata{Mari\'{e}n Abreu}
{Dipartimento di Matematica, Informatica ed Economia\\ Universit\`{a} degli Studi della Basilicata, Italy}
{marien.abreu@unibas.it}
{}

\authordata{John Baptist Gauci}
{Department of Mathematics, University of Malta, Malta}
{john-baptist.gauci@um.edu.mt}
{}

\authordata{Jean Paul Zerafa}
{Dipartimento di Scienze Fisiche, Informatiche e Matematiche\\ Universit\`{a} degli Studi di Modena e Reggio Emilia, Italy;\\
Department of Technology and Entrepreneurship Education\\ University of Malta, Malta}
{jean-paul.zerafa@um.edu.mt}
{}


\begin{abstract}
The rook graph is a graph whose edges represent all the possible legal moves of the rook chess piece on a chessboard. The problem we consider is the following. Given any set $M$ containing pairs of cells such that each cell of the $m_1 \times m_2$ chessboard is in exactly one pair, we determine the values of the positive integers $m_1$ and $m_2$ for which it is possible to construct a closed tour of all the cells of the chessboard which uses all the pairs of cells in $M$ and some edges of the rook graph. This is an alternative formulation of a graph-theoretical problem presented in [\emph{Electron. J. Combin.} \textbf{28(1)} (2021), \#P1.7] involving the Cartesian product $G$ of two complete graphs $K_{m_1}$ and $K_{m_2}$, which is, in fact, isomorphic to the $m_{1}\times m_{2}$ rook graph. The problem revolves around determining the values of the parameters $m_1$ and $m_2$ that would allow any  perfect matching of the complete graph on the same vertex set of $G$ to be extended to a Hamiltonian cycle by using only edges in $G$.
\end{abstract}
\end{frontmatter}

\section{Introduction}
The rook chess piece is allowed to move in a horizontal and vertical manner only---no diagonal moves are permissible. The rook graph represents all the possible moves of a rook on a chessboard, with its vertices and edges corresponding to the cells of the chessboard, and the legal moves of the rook from one cell to the other, respectively.
\begin{figure}[h]
      \centering
      \includegraphics[width=0.275\textwidth]{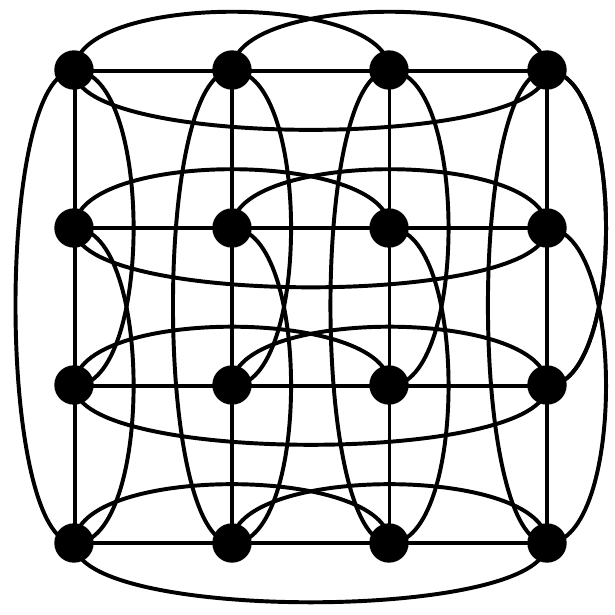}
      \caption{The $4\times 4$ rook graph isomorphic to $K_{4}\square K_{4}$}
      \label{FigureKm1m2}
\end{figure}
All the legal moves of a rook on a $m_{1}\times m_{2}$ chessboard give rise to the \emph{$m_{1}\times m_{2}$ rook graph}. In what follows we consider the following problem.

\begin{problem}\label{problem main}
Let $G$ be a $m_{1} \times m_{2}$ chessboard and let $M$ be a set containing pairs of distinct cells of $G$ such that each cell of $G$ belongs to exactly one pair in $M$. Determine the values of $m_{1}$ and $m_{2}$ for which it is possible to construct a closed tour $H$ visiting all the cells of the chessboard $G$ exactly once, such that:
\begin{itemize}
\item[(i)] consecutive cells in $H$ are either a pair of cells in $M$, or two cells in $G$ which can be joined by a legal rook move; and 
\item[(ii)] $H$ contains all pairs of cells in $M$.
\end{itemize} 
\end{problem}

In other words, given any possible choice of a set $M$ as defined above, is a rook good enough to let one visit, exactly once, all the cells on a chessboard and finish at the starting cell, in such a way that each pair of cells in $M$ is allowed to and must be used once? We remark that $M$ can contain pairs of cells which are not joined by a legal rook move.

As many other mathematical chess problems, the above problem can be restated in graph theoretical terms (for a detailed exposition, we suggest the reader to \cite{Schwenk}). We first give some definitions, and for definitions and notation not explicitly stated here, we refer the reader to \cite{Diestel}. All graphs considered in the sequel will be simple, that is, loops and multiple edges are not allowed. For any graph $G$ with vertex set $V(G)$ and edge set $E(G)$, we let $K_{G}$ denote the complete graph on the same vertex set $V(G)$ of $G$. Let $G$ be of even order, that is, having an even number of vertices. A \emph{Hamiltonian cycle} of a graph $G$ is a cycle of $G$ which visits every vertex of $G$. A \emph{perfect matching} $N$ of a graph $G$ is a set of edges of $G$ such that every vertex of $G$ belongs to exactly one edge in $N$. This means that no two edges in $N$ have a common vertex and that $N$ is a set of independent edges covering $V(G)$. Let $G$ be a graph of even order. A Hamiltonian cycle of $G$ can be considered as the disjoint union of two perfect matchings of $G$. A perfect matching of $K_{G}$ is said to be a \emph{pairing} of $G$. In what follows we shall consider Hamiltonian cycles of $K_{G}$ (for some graph $G$ of even order) composed of a pairing of $G$ and a perfect matching of $G$. In order to distinguish between pairings of $G$, which may possibly contain edges not in $G$, and perfect matchings of $G$, we shall depict pairing edges as green, bold and dashed, and edges of a perfect matching of $G$ as black and bold. 
To emphasise that pairings can contain edges in $G$, we shall depict such edges with a black thin line underneath the green, bold and dashed edge described above. This can be clearly seen in Figure \ref{FigureCube}.

\begin{figure}[h]
      \centering
      \includegraphics[width=0.5\textwidth]{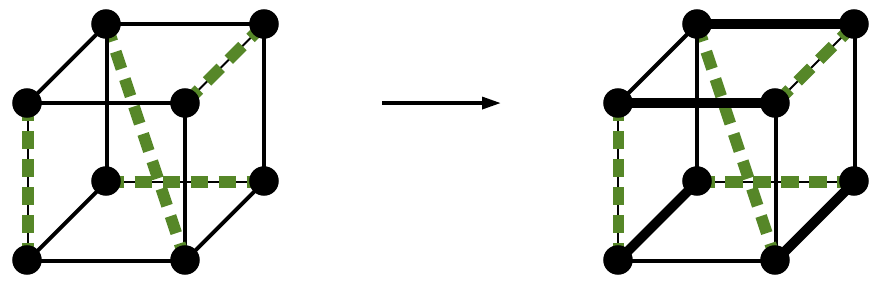}
      \caption{A pairing $M$ in the cube $\mathcal{Q}_{3}$ which is not a perfect matching of $\mathcal{Q}_{3}$ and a Hamiltonian cycle of $K_{\mathcal{Q}_{3}}$ containing $M$}
      \label{FigureCube}
\end{figure} In 2015, the authors in \cite{ThomassenEtAl} say that a graph $G$ has the \emph{Pairing-Hamiltonian property} (the \emph{PH-property} for short) if every pairing $M$ of $G$ can be extended to a Hamiltonian cycle $H$ of $K_{G}$ in which $E(H)-M\subseteq E(G)$. If a graph has the PH-property, for simplicity we shall sometimes say that the graph is PH. 
In order to provide the reader with some examples of graphs having the PH-property, we remark that the authors in \cite{ThomassenEtAl}, amongst other results, gave a complete characterisation of the cubic graphs, that is, graphs with all vertices having degree 3, having the PH-property. There are only three: the complete graph $K_{4}$, the complete bipartite graph $K_{3,3}$ and the 3-dimensional cube $\mathcal{Q}_{3}$ (depicted in Figure \ref{FigureCubicPH}). We note that in the first diagram of Figure \ref{FigureCube}, one of the green, bold and dashed edges is not an edge of $\mathcal{Q}_3$, and thus the diagram illustrates a possible pairing of $\mathcal{Q}_3$ which is not a perfect matching of $\mathcal{Q}_3$. As shown in Figure \ref{FigureCube}, this pairing can be extended to a Hamiltonian cycle of $\mathcal{Q}_3$ by using edges of $\mathcal{Q}_3$. The same argument can be repeated for all pairings of the three graphs shown in Figure \ref{FigureCubicPH}; hence why they have the PH-property. A similar property to the PH-property is the \emph{PMH-property}, short for the \emph{Perfect-Matching-Hamiltonian property} (see \cite{PMHAbreuEtAl} for a more detailed introduction). A graph is said to have the PMH-property, if every perfect matching $M$ of $G$ can be extended to a Hamiltonian cycle $H$ of $K_{G}$ in which $E(H)-M\subseteq E(G)$. We note that in this case, $H$ would also be a Hamiltonian cycle of $G$ itself. In other words, the PMH-property is equivalent to the PH-property restricted to pairings of $G$ which are also perfect matchings of $G$. Thus, the PMH-property is a somewhat weaker property than the PH-property. 

\begin{figure}[h]
      \centering
      \includegraphics[width=0.4\textwidth]{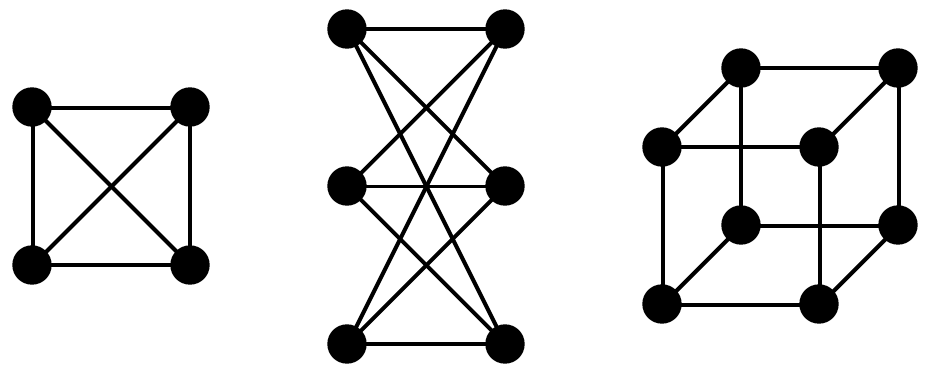}
      \caption{The only cubic graphs having the PH-property}
      \label{FigureCubicPH}
\end{figure}

The \emph{Cartesian product} $G\square H$ of two graphs $G$ and $H$ is a graph whose vertex set is the Cartesian product $V(G) \times V(H)$ of $V(G)$ and $V(H)$. Two vertices $(u_i,v_j)$ and $(u_k,v_l)$ are adjacent precisely if  $u_i=u_k$ and $v_jv_l\in E(H)$ or $u_iu_k \in E(G)$ and $v_j=v_l$. Thus, 
\begin{eqnarray*}
& V(G\square H)= \{(u_r,v_s) : u_r \in V(G) \text{ and } v_s \in V(H)\},\text{ and } \\
& E(G\square H)=\{(u_i,v_j)(u_k,v_l):u_i=u_k,v_jv_l\in E(H)\text{ or } u_iu_k \in E(G), v_j=v_l\}.
\end{eqnarray*}
The \emph{$m_{1}\times m_{2}$ rook graph} is in fact isomorphic to the Cartesian product of the complete graphs $K_{m_{1}}$ and $K_{m_{2}}$, denoted by $K_{m_{1}}\square K_{m_{2}}$.
 
Another result in \cite{ThomassenEtAl} which we shall also be using later on is the following.

\begin{theorem}[Alahmadi \emph{et al.} \cite{ThomassenEtAl}]\label{TheoremThom}
The Cartesian product of a complete graph $K_m$ ($m$ even and $m\geq 6$) and a path $P_q
$ ($q\geq 1$) has the PH-property.
\end{theorem}

However, this was not the first time that pairings extending to Hamiltonian cycles were studied. In 2007, Fink \cite{Fink} proved what we believe is one of the most significant results in this area so far: for every $n\geq 2$, the $n$-dimensional hypercube is PH, thus answering a conjecture made by Kreweras (see \cite{kreweras}). The proof of this result, although technical, is very short and elegant. 

With these notions in place, we can restate the above problem as follows.
\begin{problem}[Problem \ref{problem main} restated]\label{problem restated}
Let $G$ be the $m_{1} \times m_{2}$ rook graph, or equivalently $K_{m_{1}}\square K_{m_{2}}$. Determine for which values of $m_{1}$ and $m_{2}$ does $G$ have the PH-property.
\end{problem}

Clearly, in order for $K_{m_{1}}\square K_{m_{2}}$ to admit a pairing, at least one of $m_{1}$ and $m_{2}$ must be even, and without loss of generality, in the sequel we shall tacitly assume that $m_{1}$ is even.

We recall that the \emph{line graph} $L(G)$ of a graph $G$ is the graph whose vertices correspond to the edges of $G$, and two vertices of $L(G)$ are adjacent if the corresponding edges in $G$ are incident to a common vertex. The \emph{$m_{1}\times m_{2}$ rook graph}. or equivalently $K_{m_{1}}\square K_{m_{2}}$,  can also be seen as the line graph of the complete bipartite graph $K_{m_{1},m_{2}}$. The authors in \cite{PMHAbreuEtAl} give some sufficient conditions for a graph $G$ in order to guarantee that its line graph $L(G)$ has the PMH-property. Amongst other results, they show that the line graph of complete graphs $K_{n}$, for $n\equiv 0,1\pmod{4}$, has the PMH-property, and that, by a similar reasoning, $L(K_{m,m})$ has the PMH-property for every even $m\geq 50$. In Section \ref{section main}, we determine for which values $m_{1}$ and $m_{2}$ (with $m_{1}$ not necessarily equal to $m_{2}$) does $L(K_{m_{1},m_{2}})$ admit not only the PMH-property, but also the PH-property. This gives a complete solution to Problem \ref{problem restated}.

\section{Main result}\label{section main}

In this section we give a complete solution to Problem \ref{problem restated}, summarised in the following theorem.

\begin{theorem}\label{TheoremMain}
Let $m_{1}$ be an even integer and let $m_{2}\geq 1$. The $m_{1} \times m_{2}$ rook graph does not have the PH-property if and only if $m_{1}=2$ and $m_{2}$ is odd.
\end{theorem}

\begin{proof}
When $m_{2}=1$, $K_{m_{1}}\square K_1$ is $K_{m_{1}}$ and the result clearly follows. Consequently, we shall assume that $m_{2}>1$. By Theorem \ref{TheoremThom}, $K_{m_{1}}\square K_{m_{2}}$ is PH when $m_{1}\geq 6$, since $K_{m_{1}}\square K_{m_{2}}$ contains $K_{m_{1}}\square P_{m_{2}}$, and, in general, if a graph contains a spanning subgraph which is PH, the initial graph is itself PH.

So consider the cases when $m_{1}=2$ or $4$. If $m_{1}=2$, $K_{m_{1}}\square K_{m_{2}}$ is PH if and only if $m_{2}\equiv 0 \pmod{2}$. In fact, if $m_{2}$ is odd, the pairing consisting of the $m_{2}$-edge-cut between the two copies of $K_{m_2}$ cannot be extended to a Hamiltonian cycle, as can be seen in Figure \ref{FigureK23}. 
\begin{figure}[h!]
      \centering
      \includegraphics[width=0.165\textwidth]{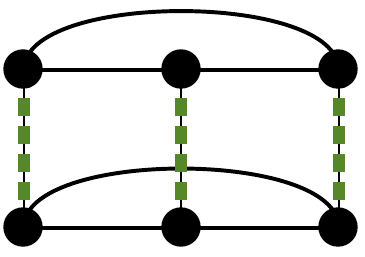}
      \caption{A pairing in $K_{2}\square K_{3}$ which cannot be extended to a Hamiltonian cycle}
      \label{FigureK23}
\end{figure}
If $m_{2}$ is even, the result follows once again by Theorem \ref{TheoremThom} when $m_{2}\geq 6$. If $m_{2}=2$, the result easily follows, and when $m_{2}=4$, $K_{2}\square K_{4}$ is PH because the 3-dimensional cube $\mathcal{Q}_{3}$ is a subgraph of $K_{2}\square K_{4}$ and has the PH-property by Fink's result in \cite{Fink} (also referred to previously).

What remains to be considered is the case when $m_1=4$ and $m_2\geq 3$. The graph $K_4\square K_4$ contains $C_{4}\square C_{4}$, the 4-dimensional hypercube $Q_{4}$, which is PH (\cite{Fink}), and for $m_{2}\geq 6$ and $m_{2}$ even, the result follows once again by Theorem \ref{TheoremThom}. Therefore, what remains to be shown is the case when $m_2 \geq 3$ and $m_2$ is odd, which is settled in the following technical lemma.
\end{proof}

\begin{lemma}\label{LemmaK4odd}
For every odd $m\geq 3$, the $4 \times m$ rook graph has the PH-property.
\end{lemma}

\begin{proof}
Let the $4\times m$ rook graph $K_{4}\square K_{m}$ be denoted by $G$. We let the vertex set of $G$ be $\{a_{i},b_{i},c_{i},d_{i}:i\in[m]\}$, such that for each $i$, the vertices $a_{i},b_{i},c_{i},d_{i}$ induce a complete graph on four vertices, denoted by $K_{4}^{i}$, and the vertices represented by the same letter induce a $K_{m}$. Let $M$ be a pairing of $G$. We consider two cases: 

\noindent\textbf{Case 1.} $M$ does not induce a perfect matching in each $K_{4}^{i}$; and 

\noindent\textbf{Case 2.} $M$ induces a perfect matching in each $K_{4}^{i}$.

We start by considering \textbf{Case 1}, and without loss of generality assume that $|M\cap E(K_{4}^{1})|<2$. If we delete all the edges having exactly one end-vertex in $K_{4}^{1}$ from $G$, we obtain two components $G_{1}$ and $G_{2}$ isomorphic to $K_{4}^{1}$ and $K_{4}\square K_{m-1}$, respectively. Since $G_{1}$ is of even order and $M\cap E(G_{1})$ is not a perfect matching of this graph, $G_{1}$ has an even number (two or four) of vertices which are unmatched by $M\cap E(G_{1})$.

We pair these unmatched vertices such that $M\cap E(G_{1})$ is extended  to a perfect matching $M_{1}$ of $G_{1}$. By a similar reasoning, $M\cap E(G_{2})$ does not induce a pairing of $G_{2}$ and the number of vertices in $G_{2}$ which are unmatched by $M\cap E(G_{2})$ is again two or four. Without loss of generality, let $a_{1},b_{1}$ be two vertices in $G_{1}$ unmatched by $M\cap E(G_{1})$ such that $a_{1}b_{1}\in M_{1}$, and let $x,y$ be the two vertices in $G_{2}$ such that $a_{1}x$ and $b_{1}y$ are both edges in the pairing $M$ of $G$. We extend $M\cap E(G_{2})$ to a pairing $M_{2}$ of $G_{2}$ by adding the edge $xy$ to $M\cap E(G_{2})$, and we repeat this procedure until all vertices in $G_{2}$ are matched. Since $m-1$ is even, $G_{2}$ has the PH-property and so $M_{2}$ can be extended to a Hamiltonian cycle $H_{2}$ of $K_{G_{2}}$. We extend $H_{2}$ to a Hamiltonian cycle of $G$ containing $M$ as follows. If $c_{1}d_{1}\in M\cap E(G_{1})$, we replace the edge $xy$ in $H_{2}$ by the edges $xa_{1}, a_{1}d_{1}, d_{1}c_{1}, c_{1}b_{1},b_{1}y$, as in Figure \ref{FigureK43}.
Otherwise, $c_{1}d_{1}\in M_{1}-(M\cap E(G_{1}))$, and so there exist two vertices $u,v$ in $G_{2}$ such that $c_{1}u$ and $d_{1}v$ belong to belong to the initial pairing $M$, and $uv$ belongs to $M_{2}$. In this case, we replace the edges $xy$ and $uv$ in $H_{2}$ by the edges $xa_{1}, a_{1}b_{1}, b_{1}y$, and $uc_{1}, c_{1}d_{1}, d_{1}v$, respectively. In either case, $H_{2}$ is extended to a Hamiltonian cycle of $G$ containing the pairing $M$, as required. 

\begin{figure}[h]
      \centering
      \includegraphics[width=0.55\textwidth]{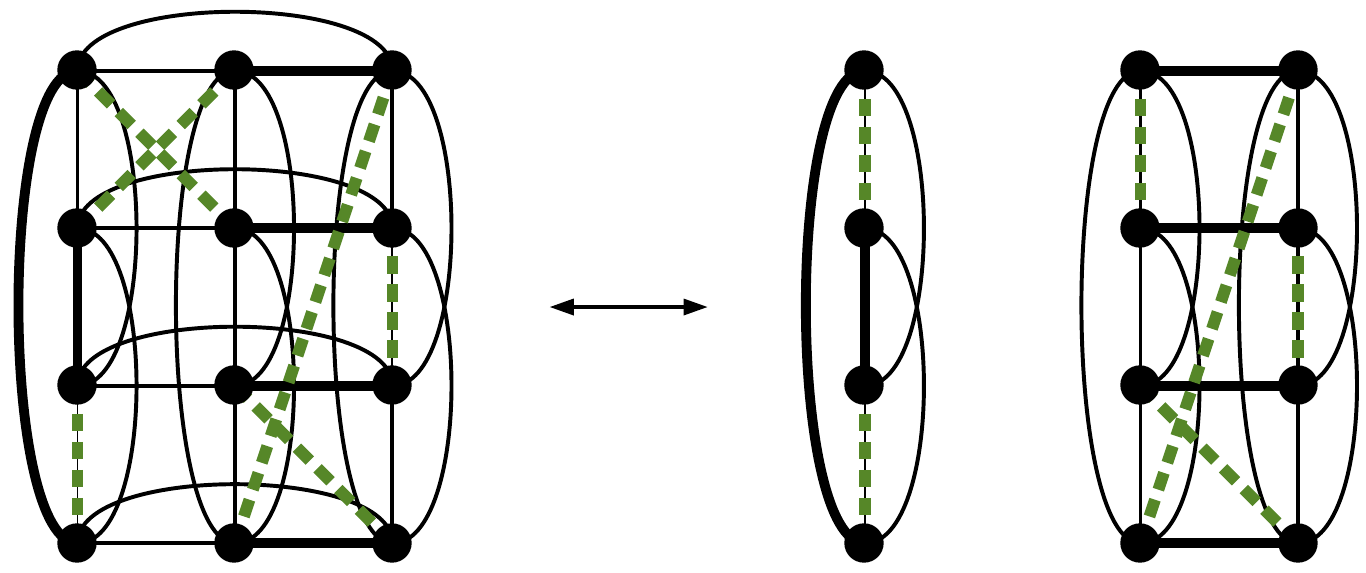}
      \caption{An illustration of the inductive step in Case 1 when $m_{2}=3$}
      \label{FigureK43}
\end{figure}

Next, we move on to \textbf{Case 2}, that is, when $M$ induces a perfect matching in each $K_{4}^{i}$. This case is true by Proposition 1 in \cite{ThomassenEtAl}, however, here we adopt a constructive and more detailed approach highlighting the very useful technique used in \cite{Fink}. There are three different ways how $M$ can intersect the edges of $K_{4}^{i}$, namely $M\cap E(K_{4}^{i})$ can either be equal to $\{a_{i}b_{i},c_{i}d_{i}\}$, $\{a_{i}c_{i},b_{i}d_{i}\}$, or $\{a_{i}d_{i},b_{i}c_{i}\}$. The number of 4-cliques intersected by $M$ in $\{a_{i}b_{i},c_{i}d_{i}\}$ is denoted by $\fourIdx{}{}{ab}{cd}{\nu}$, and we shall define $\fourIdx{}{}{ac}{bd}{\nu}$ and $\fourIdx{}{}{ad}{bc}{\nu}$ in a similar way. Without loss of generality, we shall assume that $\fourIdx{}{}{ab}{cd}{\nu}\geq \fourIdx{}{}{ac}{bd}{\nu} \geq \fourIdx{}{}{ad}{bc}{\nu}$. We shall also assume that the first $\fourIdx{}{}{ab}{cd}{\nu}$ 4-cliques in $\{K_{4}^{i}:i\in[m]\}$ are the ones intersected by $M$ in $\{a_{i}b_{i},c_{i}d_{i}\}$, and, if $\fourIdx{}{}{ad}{bc}{\nu}\neq 0$, the last $\fourIdx{}{}{ad}{bc}{\nu}$ 4-cliques are the ones intersected by $M$ in $\{a_{i}d_{i},b_{i}c_{i}\}$. This can be seen in Figure \ref{Figure221}, in which ``unnecessary" curved edges of $G$ are not drawn so as to render the figure more clear.

\begin{figure}[h]
      \centering
      \includegraphics[width=0.512\textwidth]{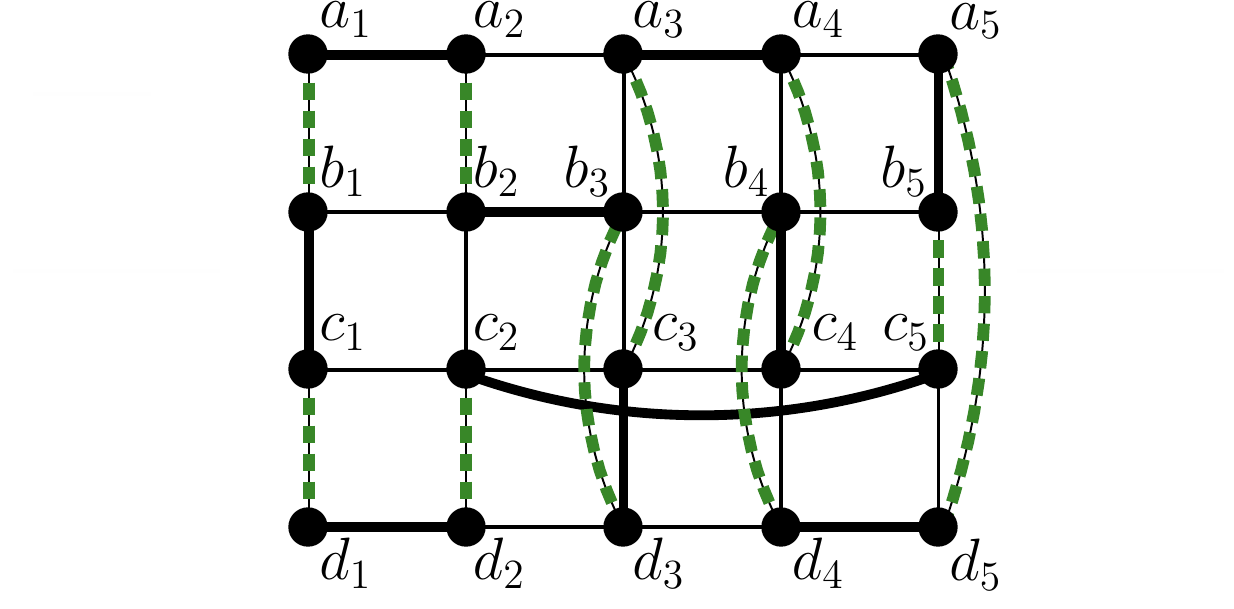}
      \caption{$G$ when $\fourIdx{}{}{ab}{cd}{\nu}=2$, $\fourIdx{}{}{ac}{bd}{\nu}=2$ and $\fourIdx{}{}{ad}{bc}{\nu}=1$}
      \label{Figure221}
\end{figure}

When $\fourIdx{}{}{ab}{cd}{\nu}=1$, we have that $\fourIdx{}{}{ac}{bd}{\nu}=\fourIdx{}{}{ad}{bc}{\nu}=1$, and in this case it is easy to see that $M$ can be extended to a Hamiltonian cycle of $K_G$, for example $(a_1, b_1, c_1, d_1, d_3, a_3, c_3,  b_3, b_2, d_2, \linebreak c_2, a_2)$. We remark that this is the only time when all the 4-cliques are intersected differently by $M$. Therefore, assume $\fourIdx{}{}{ab}{cd}{\nu}\geq 2$. First, let $\fourIdx{}{}{ab}{cd}{\nu}=2$. If $\fourIdx{}{}{ad}{bc}{\nu}=0$, then, $\fourIdx{}{}{ac}{bd}{\nu}=1$ and it is easy to see that $M$ can be extended to a Hamiltonian cycle of $K_{G}$, for example $(a_1, b_1, b_2, a_2, a_3, c_3, b_3, d_3, d_2, c_2, c_1, d_1)$. The only other possibility is to have $\fourIdx{}{}{ac}{bd}{\nu}=2$ and $\fourIdx{}{}{ad}{bc}{\nu}=1$, and once again $M$ can be extended to a Hamiltonian cycle of $K_{G}$, as Figure \ref{Figure221} shows.

Thus, we can assume that $\fourIdx{}{}{ab}{cd}{\nu}\geq 3$. Let $r=\fourIdx{}{}{ab}{cd}{\nu}+\fourIdx{}{}{ac}{bd}{\nu}$ and let $r'$ be the largest even integer less than or equal to $r$. Moreover, let $G_{1}$ be the subgraph of $G$ induced by the vertices $\{b_{i},c_{i}: i\in[m]\}$ (isomorphic to $K_{2}\square K_{m}$) and let $M_{1}=\{b_{1}b_{2},\ldots, b_{r'-1}b_{r'}, c_{1}c_{2},\ldots, \linebreak c_{r'-1}c_{r'}, b_{r'+1}c_{r'+1},\ldots, b_{m}c_{m}\}$. Clearly, $M_{1}$ is a pairing of $G_{1}$ which contains $M\cap E(G_{1})$, and can be extended to a Hamiltonian cycle $H_{1}$ of $K_{G_{1}}$ as follows: $(b_{1},b_{2},\ldots,\linebreak  b_{r'}, b_{r'+1}, c_{r'+1},c_{r'+2},b_{r'+2},\ldots, b_{m}c_{m},c_{r'}, c_{r'-1}, \ldots,c_{1})$. This is depicted in Figure \ref{FigureH2}. We note that if $r'=m-1$, we do not consider the index $r'+2$ in the last sequence of vertices forming $H_{1}$.
Deleting the edges belonging to $M_{1}-M$ from $H_{1}$ gives a collection of $r$ disjoint paths $\mathcal{P}=\{P^{i}:i\in[r]\}$. We note that the union of all the end-vertices of the paths in $\mathcal{P}$ give $\{b_{i},c_{i}:i\in[r]\}$. If we look at the example given in Figure \ref{FigureH2}, the only path in $\mathcal{P}$ on more than two vertices is the path $b_{8}b_{9}c_{9}c_{10}b_{10}b_{11}c_{11}c_{8}$.

\begin{figure}[h]
      \centering
      \includegraphics[width=.75\textwidth]{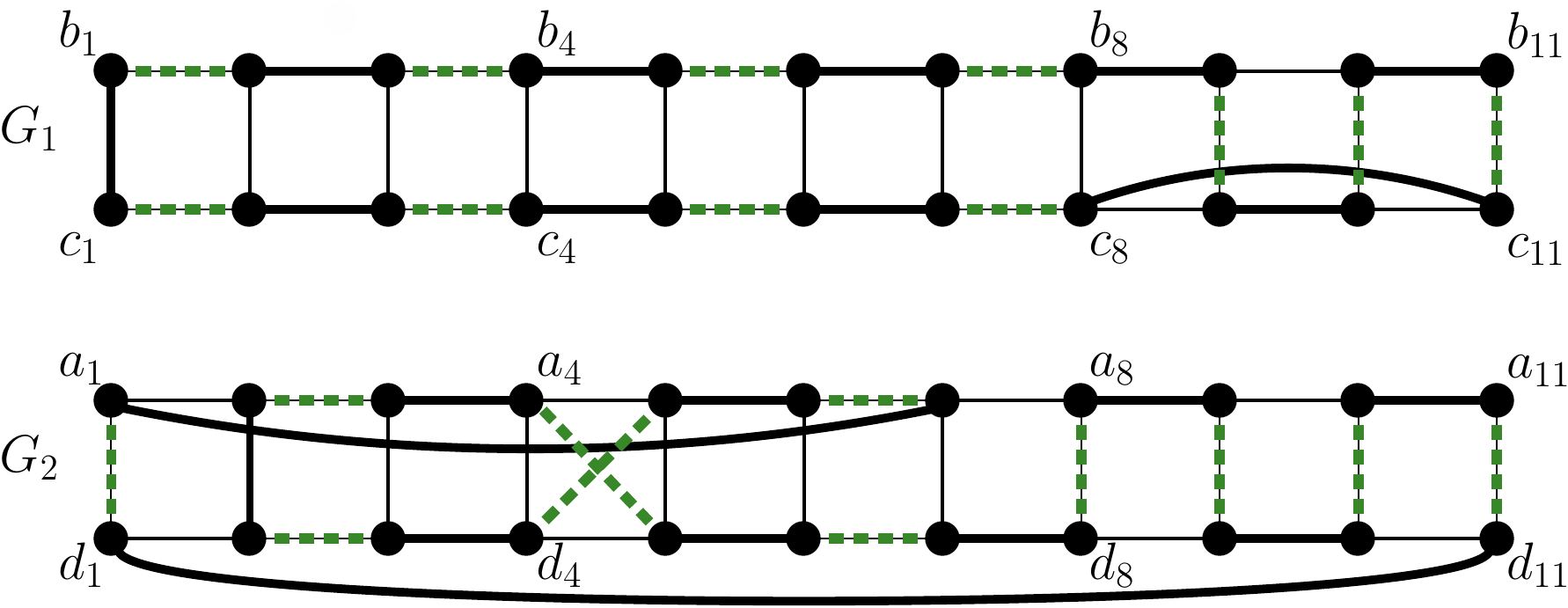}
      \caption{$G_{1}$ and $G_{2}$ when $\fourIdx{}{}{ab}{cd}{\nu}=4$, $r=r'=8$, and $m=11$ in Case 2}
      \label{FigureH2}
\end{figure}

Next, let $G_{2}$ be the subgraph of $G$ induced by the vertices $\{a_{i},d_{i}: i\in[m]\}$, which is isomorphic to $K_{2}\square K_{m}$ as $G_{1}$. For every $i\in[r]$, we let $u_{i}$ and $v_{i}$ be the two end-vertices of the path $P^{i}$, and we let $x_{i}$ and $y_{i}$ be the two vertices in $G_{2}$ such that $u_{i}x_{i}$ and $v_{i}y_{i}$ both belong to $M$. We remark that $\{a_{i},d_{i}:i\in[r]\}=\{x_{i},y_{i}:i\in[r]\}$. Let $M_{2}=\{x_{1}y_{1}, \ldots, x_{r}y_{r}\}\cup (M\cap E(G_{2}))$. If $r=m$, then $M\cap E(G_{2})$ is empty, otherwise it consists of $\{a_{r+1}d_{r+1},\ldots, a_{m}d_{m}\}$. 
If $\fourIdx{}{}{ab}{cd}{\nu}$ is even (as in Figure \ref{FigureH2}), $M_{2}$ contains: \[\{a_{1}d_{1}, a_{2}a_{3}, \ldots, a_{\fourIdx{}{}{ab}{cd}{\nu}-2}a_{\fourIdx{}{}{ab}{cd}{\nu}-1}, a_{\fourIdx{}{}{ab}{cd}{\nu}}d_{\fourIdx{}{}{ab}{cd}{\nu}+1}, d_{2}d_{3}, \ldots, d_{\fourIdx{}{}{ab}{cd}{\nu}-2}d_{\fourIdx{}{}{ab}{cd}{\nu}-1}, d_{\fourIdx{}{}{ab}{cd}{\nu}}a_{\fourIdx{}{}{ab}{cd}{\nu}+1}\}.\] 

Otherwise, $M_{2}$ contains $\{a_{1}d_{1}, a_{2}a_{3}, \ldots, a_{\fourIdx{}{}{ab}{cd}{\nu}-1}a_{\fourIdx{}{}{ab}{cd}{\nu}}, d_{2}d_{3}, \ldots, d_{\fourIdx{}{}{ab}{cd}{\nu}-1}d_{\fourIdx{}{}{ab}{cd}{\nu}}\}$. Moreover, if $r$ is even, then $a_{r}d_{r}\in M_{2}$. In either case, $M_{2}$ can be extended to a Hamiltonian cycle $H_{2}$ of $K_{G_{2}}$, as can be seen in Figure \ref{FigureH2}, which shows the case when $\fourIdx{}{}{ab}{cd}{\nu}$ and $r$ are both even. We remark that the green, bold and dashed edges in the figure are the ones in $M_{1}$ and $M_{2}$. If for each $i\in[r]$, we replace the edges $x_{i}y_{i}$ in $H_{2}$ by $x_{i}u_{i}$, the path $P^{i}$, and $v_{i}y_{i}$ (as in Figure \ref{FigureH2FINAL}), a Hamiltonian cycle of $K_{G}$ containing $M$ is obtained, proving our theorem.
\end{proof}

\begin{figure}[h]
      \centering
      \includegraphics[width=0.7\textwidth]{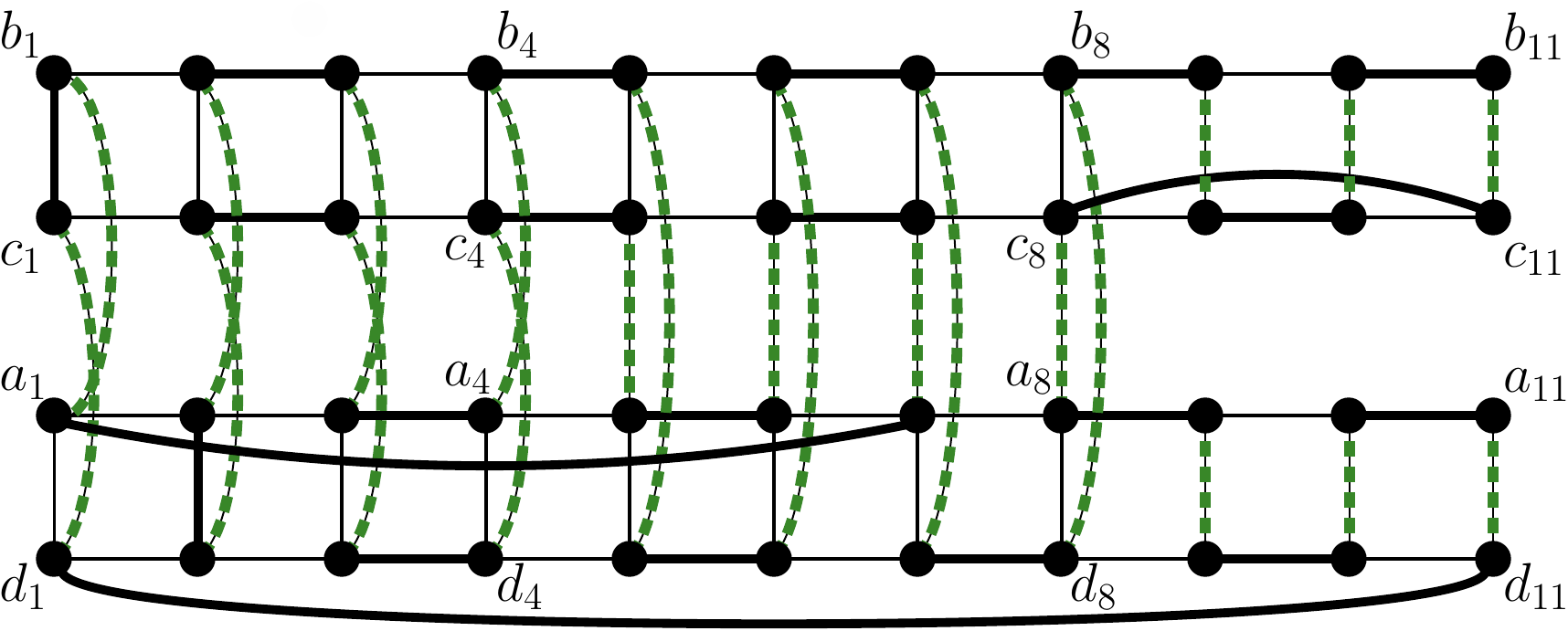}
      \caption{Extending $H_{1}$ and $H_{2}$ from Fig. \ref{FigureH2} to a Hamiltonian cycle of $K_{G}$ containing $M$}
      \label{FigureH2FINAL}
\end{figure}

\section{Bishop-on-a-rook graph}

In the next theorem we present a rather simple proof to show that the complete bipartite graph having equal partite sets (otherwise it does not admit a perfect matching) is PH.

\begin{theorem}\label{TheoremKnnPH}
For every $n\geq 2$, the complete bipartite graph $K_{n,n}$ has the PH-property.
\end{theorem}
\begin{proof}
Let $\{u_{1},\ldots,u_{n}\}$ and $\{w_{1},\ldots,w_{n}\}$ be the partite sets of $K_{n,n}$. We proceed by induction on $n$. When $n=2$, result holds since $K_{2,2}\simeq K_{2}\square K_{2}$. So assume $n>2$ and let $M$ be a pairing of $K_{n,n}$. If $M=\{u_{i}w_{i}:i\in[n]\}$, then $M$ easily extends to a Hamiltonian cycle of the underlying complete graph on $2n$ vertices. Thus, assume there exists $j\in[n]$ such that $u_{j}w_{j}\not\in M$. Without loss of generality, let $j$ be equal to $n$. Then, $M$ contains the edges $xu_{n}$ and $yw_{n}$, for some $x$ and $y$ belonging to the set $Z=\{u_{i}, w_{i}: i\in[n-1]\}$. We note that $Z$ induces the complete bipartite graph $K_{n-1,n-1}$ with partite sets $\{u_1, \ldots, u_{n-1}\}$ and $\{w_1, \ldots, w_{n-1}\}$, which we denote by $G'$.  The set of edges $M'=M\cup xy- xu_{n}-yw_{n}$ is a pairing of $G'$, and so, by induction on $n$, $M'$ can be extended to a Hamiltonian cycle $H'$ of $K_{G'}$. This Hamiltonian cycle can be extended to a Hamiltonian cycle $H$ of the underlying complete graph of $K_{n,n}$ by replacing the edge $xy$ in $H'$, by the edges $xu_{n},u_{n}w_{n},w_{n}y$. The resulting Hamiltonian cycle $H$ clearly contains $M$, proving our theorem.
\end{proof}

Although the statement and proof of Theorem \ref{TheoremKnnPH} are quite easy, they may lead to another intriguing problem. From Theorem \ref{TheoremMain} we know that the rook  is not good enough to solve our problem on a $2\times m_{2}$ chessboard when $m_{2}$ is odd. However, the above result shows that if the rook was somehow allowed to do only vertical and diagonal moves (instead of vertical and horizontal moves only), then it would always be possible to perform a closed tour on a $2 \times m_{2}$ chessboard in such a way that each pair of cells in $M$ is allowed to and must be used once, no matter the choice of $M$.
We shall call this new hybrid chess piece the \emph{bishop-on-a-rook}, and, as already stated, it is only allowed to move in a vertical and diagonal manner---no horizontal moves are permissible. As in the case of the rook, all the legal moves of a bishop-on-a-rook on a $m_{1}\times m_{2}$ chessboard give rise to the $m_{1}\times m_{2}$ \emph{bishop-on-a-rook graph}, with $m_{1}$ corresponding to the vertical axis.

As before, for the $m_1 \times m_2$ bishop-on-a-rook graph to be PH, at least one of $m_1$ or $m_2$ must be even. Moreover, we remark that when $m_{2}\leq m_{1}$, the $m_{1}\times m_{2}$ bishop-on-a-rook graph contains $K_{m_{1}}\square K_{m_{2}}$ as a subgraph. Finally, we also observe that the $m_{1}\times m_{2}$ bishop-on-a-rook graph is isomorphic to the co-normal product of $K_{m_{1}}$ and $\overline{K}_{m_2}$, where the latter is the empty graph on $m_{2}$ vertices. 
The \emph{co-normal product} $G*H$ of two graphs $G$ and $H$ is a graph whose vertex set is the Cartesian product $V(G) \times V(H)$ of $V(G)$ and $V(H)$, and two vertices $(u_i,v_j)$ and $(u_k,v_l)$ are adjacent precisely if  $u_iu_k\in E(G)$ or $v_jv_l\in E(H)$. Thus, 
\begin{eqnarray*}
& V(G*H)= \{(u_r,v_s) : u_r \in V(G) \text{ and } v_s \in V(H)\},\text{ and } \\
& E(G*H)=\{(u_i,v_j)(u_k,v_l):u_iu_k\in E(G)\text{ or } v_jv_l\in E(H)\}.
\end{eqnarray*}

We wonder for which values $m_{1}$ and $m_{2}$ is the $m_{1}\times m_{2}$ bishop-on-a-rook graph PH.

\end{document}